\documentclass[12pt,leqno,a4paper]{amsart}
\usepackage{amssymb}
\usepackage{enumitem}

\newcommand{\FF}{{\mathbb{F}}}
\newcommand{\ZZ}{\mathbb{Z}}

\newcommand{\bG}{{\mathbf{G}}}
\newcommand{\bL}{{\mathbf{L}}}
\newcommand{\bM}{{\mathbf{M}}}
\newcommand{\bP}{{\mathbf{P}}}
\newcommand{\bU}{{\mathbf{U}}}

\newcommand{\cF}{{\mathcal{F}}}
\newcommand{\cH}{{\mathcal{H}}}
\newcommand{\cO}{{\mathcal{O}}}
\newcommand{\cS}{{\mathcal{S}}}

\newcommand{\AV}{{\operatorname{AV}}}
\newcommand{\rk}{{\operatorname{rk}}}
\newcommand{\IBr}{{\operatorname{IBr}}}
\newcommand{\Ind}{{\operatorname{Ind}}}
\newcommand{\Irr}{{\operatorname{Irr}}}
\newcommand{\GL}{{\operatorname{GL}}}
\newcommand{\PGL}{{\operatorname{PGL}}}

\newcommand{\SO}{{\operatorname{SO}}}
\newcommand{\Spin}{{\operatorname{Spin}}}
\newcommand{\Sp}{{\operatorname{Sp}}}
\newcommand{\ad}{{\operatorname{ad}}}
\newcommand{\mi}{{\operatorname{min}}}
\newcommand{\sgn}{{\operatorname{sgn}}}
\newcommand\RLG{{R_L^G}}

\newcommand{\tw}[1]{{}^{#1}\!}
\newcommand{\hlf}{\frac{1}{2}}

\newcommand{\Chevie}{{\texttt Chevie}{}}
\newcommand{\pl}{{\!+\!}}
\newcommand{\mn}{{\!-\!}}

\let\la=\lambda
\let\vhi=\varphi

\newtheorem{thm}{Theorem}[section]
\newtheorem{lem}[thm]{Lemma}
\newtheorem{prop}[thm]{Proposition}
\newtheorem{cor}[thm]{Corollary}

\theoremstyle{remark}
\newtheorem{rem}[thm]{Remark}

\begin{document}

\title[Bounding Harish-Chandra series]{Bounding Harish-Chandra series}

\date{\today}

\author{Olivier Dudas}
\address{Universit\'e Paris Diderot, UFR de Math\'ematiques,
B\^atiment Sophie Germain, 5 rue Thomas Mann, 75205 Paris CEDEX 13, France.}
\email{olivier.dudas@imj-prg.fr}

\author{Gunter Malle}
\address{FB Mathematik, TU Kaiserslautern, Postfach 3049,
         67653 Kaisers\-lautern, Germany.}
\email{malle@mathematik.uni-kl.de}

\thanks{The first author gratefully acknowledges financial support by
the ANR grant GeRepMod.
The second author gratefully acknowledges financial support by ERC
  Advanced Grant 291512 and SFB TRR 195.}



\begin{abstract}
We use the progenerator constructed in \cite{DM17} to give a necessary
condition for a simple module of a finite reductive group to be cuspidal, or
more generally to obtain information on which Harish-Chandra series it can
lie in. As a first application we show the irreducibility of the smallest
unipotent character in any Harish-Chandra series. Secondly, we determine a
unitriangular approximation to part of the unipotent decomposition matrix of
finite orthogonal groups and prove a gap result on certain Brauer character
degrees.
\end{abstract}

\maketitle


\section{Introduction} \label{sec:intro}
Let $\bG$ be a connected reductive linear algebraic group defined over a
finite field $\FF_q$, with corresponding Frobenius endomorphism $F$.
The unipotent characters of the finite reductive group $\bG^F$ were
classified by Lusztig in \cite{LuB}. One feature of this classification is that
cuspidal unipotent characters belong to a unique family, whose numerical
invariant---Lusztig's $a$-function---gets bigger as the rank of $\bG$ increases.
\smallskip

For representations in positive characteristic $\ell \nmid q$, many more
cuspidal representations can occur and they need not all lie in the same family.
The purpose of this paper is to show that nevertheless the statement about
the $a$-function still holds. In other words, we show that unipotent
representations with small $a$-value must lie in a Harish-Chandra series
corresponding to a small Levi subgroup of $G$ (see Theorem~\ref{thm:low a}).
The proof uses the progenerator constructed in \cite{DM17} which ensures
that cuspidal unipotent modules appear in the head of generalised
Gelfand--Graev representations attached to cuspidal classes.
Our result then follows from the computation of lower bounds for the numerical
invariants attached to theses classes (see Proposition~\ref{prop:bound a}).
As a consequence we deduce in Theorem~\ref{thm:HC-irr} that the
unipotent characters with smallest $a$-value in any ordinary Harish-Chandra
series remain irreducible under $\ell$-modular reduction. This generalises our
result for unitary groups from \cite{DM15}.

In Sections~\ref{sec:orth-even} and~\ref{sec:orth-odd} we apply these
considerations to Harish-Chandra series of unipotent representations of the
finite spin groups $\Spin_n^{(\pm)}(q)$ for $a$-value at most $3$
respectively~4 and determine approximations of the corresponding partial
decomposition matrices, see Theorems~\ref{thm:decD+}, \ref{thm:decD-}
and~\ref{thm:decB}. From the partial triangularity of these decomposition
matrices we then derive a gap result for the corresponding unipotent Brauer
character degrees, see Corollaries~\ref{cor:boundD} and~\ref{cor:boundB}.
\smallskip

\section{Bounding Harish-Chandra series}   \label{sec:bound}

Let $\bG$ be a connected reductive group defined over $\FF_q$, with
corresponding Frobenius endomorphism $F$. Throughout this section we assume
that $p$, the characteristic of $\FF_q$, is good for $\bG$. We let $\ell$ be
a prime different from $p$ and $(K,\cO,k)$ denote a splitting $\ell$-modular
system for $G = \bG^F$.
The irreducible representations of $kG$ are partitioned into Harish-Chandra
series, but this partition is not known in general. In this section we give
a necessary condition for a unipotent module to lie in a given series, see
Theorem~\ref{thm:low a}.
It involves a numerical invariant coming from Lusztig's $a$-function.

\subsection{Unipotent support}   \label{subsec:support}
Given $\rho\in \Irr(G)$ and $C$ an $F$-stable unipotent class of $\bG$, we
denote by $\AV(C,\rho) = |C^F|^{-1}\sum_{g \in C^F}\rho(g)$ the average value
of $\rho$ on $C^F$. We say that $C$ is a \emph{unipotent support} of $\rho$
if $C$ has maximal dimension for the property that $\AV(C,\rho)\neq 0$.
Geck \cite[Thm.~1.4]{Ge96} has shown that whenever $p$ is good for $G$ every
irreducible character $\rho$ of $G$ has a unique unipotent support, which we
will denote by $C_\rho$. By Lusztig \cite[\S11]{Lu92} (see Taylor
\cite[\S14]{Tay14} for the extension to any good
characteristic), unipotent supports of unipotent characters are \emph{special
classes} (see below). They can be computed as follows: any family $\cF$ in
the Weyl group of $\bG$ contains a unique special representation, which is
the image under the Springer correspondence of the trivial local system on a
special unipotent class $C_\cF$. Then $C_\cF$ is the common unipotent
support of all the unipotent characters in~$\cF$.

\subsection{Duality}
In \cite[III.1]{SpB} Spaltenstein studied an order-reversing map $d$ on the set
of unipotent classes of $\bG$ partially ordered by inclusion of closures.
When $p$ is good, the image of $d$ consists of the so-called special unipotent
classes, and $d$ restricts to an involution on this subset of classes.
\smallskip

The effect of $d$ on unipotent supports of unipotent characters (which
are special classes) can be computed as follows. Let $\rho$ be a unipotent
character of $G$ and let $\rho^*$ be its Alvis--Curtis dual (see
\cite[\S8]{DM91}). If $\cF$ is the family of $\Irr(W)$ attached to
$\rho$, then $\cF \otimes \sgn$ is the family attached to $\rho^*$.
Let $\phi$ be the unique special character in $\cF$. Via the Springer
correspondence, it corresponds to the trivial local system on $C_\rho$.
Then by \cite[\S3]{BV85}, the character $\phi \otimes \sgn$ corresponds
to some local system on $d(C_\rho)$. Moreover, $\phi \otimes \sgn$ is special
and therefore that local system is trivial except when $\cF$ is one
of the exceptional families in type $E_7$ and $E_8$ (see for example
\cite[\S11.3 and \S12.7]{Ca}). Consequently, for any unipotent character
$\rho$ of $G$ we have
\begin{equation}   \label{eq:(1)}
  d(C_\rho) = C_{\rho^*}
\end{equation}
This property does not hold in general for other series of characters.

\subsection{Wave front set}
Given a unipotent element $u \in G$, we denote by $\Gamma_u^G$, or
simply $\Gamma_u$, the
\emph{generalised Gelfand--Graev representation} associated with $u$. It is an
$\cO G$-lattice. The construction is given for example in \cite[\S3.1.2]{Kaw82}
(with some extra assumption on $p$) or in \cite[\S5]{Tay14}.
The first elementary properties that can be deduced are
\begin{itemize}
 \item if $\ell \neq p$, then $\Gamma_u$ is a projective $\cO G$-module;
 \item if $u$ and $u'$ are conjugate under $G$ then $\Gamma_u\cong\Gamma_{u'}$.
\end{itemize}
The character of $K\Gamma_u$ is the \emph{generalised Gelfand--Graev character}
associated with $u$. We denote it by $\gamma_u^G$, or simply $\gamma_u$. It
depends only on the $G$-conjugacy class of $u$. When $u$ is a regular unipotent
element then $\gamma_u$ is a usual Gelfand--Graev character as in
\cite[\S14]{DM91}.
\smallskip

The following theorem by Lusztig \cite[Thm.~11.2]{Lu92} and
Achar--Aubert \cite[Thm.~9.1]{AA07} (see Taylor \cite{Tay14} for
the extension to good characteristic) gives a condition on the unipotent
support of a character to occur in a generalised Gelfand-Graev character.

\begin{thm}[Lusztig, Achar--Aubert, Taylor]   \label{thm:wave}
 Let $\rho \in \Irr(G)$ and $\rho^* \in \Irr(G)$ its Alvis-Curtis dual. Then
 \begin{itemize}
  \item[\rm(a)] there exists $u\in C_{\rho^*}^F$ such that
   $\langle\gamma_u;\rho\rangle \neq 0$;
  \item[\rm(b)] if $C$ is an $F$-stable unipotent conjugacy class of $\bG$
   such that $\langle \gamma_{u};\rho\rangle \neq 0$ for some $u \in {C}^F$
   then $C \subset \overline{C_{\rho^*}}$.
 \end{itemize}
\end{thm}

Here $\overline{C}$ denotes the Zariski closure of a conjugacy class $C$.

\subsection{Lusztig's $a$-function}   \label{subsec:a-fun}
Let $\rho\in \Irr(G)$. By \cite[4.26.3]{LuB}, there exist nonnegative integers
$n_\rho, N_\rho, a_\rho$ with $n_\rho \geq 1$ and $N_\rho \equiv\pm1$ (mod $q$)
such that
$$ \dim \rho = \frac{1}{n_\rho} q^{a_\rho} N_\rho.$$
Moreover, by \cite[13.1.1]{LuB}, the integer $a_\rho$ is equal to the dimension
of the Springer fibre at any element of the unipotent support $C_\rho$ of
$\rho$. More precisely, for any $u \in C_\rho$ we have
\begin{equation}   \label{eq:(2)}
  a_\rho = \frac{1}{2}(\dim C_\bG(u) - \rk(\bG)).
  \end{equation}

Let $\chi\in\ZZ\Irr(G)$. We define $a_\chi$ to be the minimum over the
$a$-values of the irreducible \emph{unipotent} constituents of $\chi$ (and
$\infty$ if there is none). If
$\vhi\in\IBr(G)$ is unipotent we define $a_\vhi$ to be the $a$-value of its
projective cover. By extension we set $a_S = a_\vhi$ if $S$ is a simple
$kG$-module with Brauer character $\vhi$.

\subsection{Harish-Chandra series and $a$-value}
It results from Lusztig's classification of unipotent characters that there is
at most one family of $\Irr(G)$ containing a cuspidal unipotent character. In
addition, the unipotent support attached to such a family is self-dual with
respect to $d$ and has a large
$a$-value compared to the rank of the group. We give a generalisation of this
second statement to positive characteristic using the progenerator constructed
in \cite{DM17}.

For this, recall from \cite{GM96} that an $F$-stable unipotent class $C$ of
$\bG$ is said to be \emph{cuspidal} if there exists no proper $1$-split Levi
subgroup $\bL$ of $\bG$ and $u \in C \cap \bL^F$ such that the natural map
$C_\bL(u)/C_\bL(u)^\circ \rightarrow C_\bG(u)/C_\bG(u)^\circ$ is an isomorphism.
For $C$ an $F$-stable cuspidal unipotent class of $\bG$ and $d(C)$ the dual
class we set (see~(\ref{eq:(2)}))
$$a_{d(C)} := \frac{1}{2}(\dim C_\bG(u) - \rk(\bG))\qquad
  \text{where $u \in d(C)$}.$$

\begin{thm}   \label{thm:low a}
 Let $\bG$ be connected reductive. Let $S$ be a simple $k\bG^F$-module lying in
 the Harish-Chandra series above a cuspidal pair $(\bL^F,X)$ of $\bG^F$.
 Then there exists an $F$-stable unipotent class $C$
 of $\bL$ which is cuspidal for $\bL_\ad$ and such that $a_{d(C)}\le a_S$.
\end{thm}

\begin{proof}
By \cite[Thm.~2.3]{DM17}, there exists an $F$-stable unipotent class $C$ of
$\bL$ which is cuspidal for $\bL_\ad$ and $u \in C^F$ such that the generalised
Gelfand--Graev module $\Gamma_u^L$ maps onto $X$. Let $P_X$ be the projective
cover of $X$. Since $\Gamma_u^L$ is projective, it must contain $P_X$ as a
direct summand. In particular, any irreducible unipotent constituent $\chi$ of
a lift $\tilde P_X$ to $\cO$ of $P_X$ is also a constituent of $\gamma_u^L$.
For such a character $\chi$ we have $C\subseteq \overline{C_{\chi^*}}$ by
Theorem~\ref{thm:wave}. By applying Spaltenstein's duality $d$, we deduce from
\eqref{eq:(1)} that $\overline{d(C)} \supseteq d(C_{\chi^*}) = C_\chi$.
Consequently $a_{d(C)} \leq a_\chi$ for every unipotent constituent
$\chi$ of $\tilde P_X$, therefore $a_{d(C)} \leq a_X$.
\smallskip

It remains to see that $a_X \leq a_S$. This follows from the fact that
$P_S$ is a direct summand of $R_L^G(P_X)$ and that Harish-Chandra
induction can not decrease the $a$-value (see for example \cite[Cor.~8.7]{LuB}).
\end{proof}

In Sections~\ref{sec:orth-even} and~\ref{sec:orth-odd} we will apply
Theorem~\ref{thm:low a} to determine the Harish-Chandra series of unipotent
characters with small $a$-value in the finite orthogonal groups. For this, it
will be useful to derive lower bounds for $a_{d(C)}$.

\section{Cuspidal classes: minimality and induction}

\subsection{Cuspidal classes and $a$-value}
Cuspidal unipotent classes for simple groups of adjoint type in good
characteristic were classified by Geck and the second author
\cite[Prop.~3.6]{GM96}. From the classification we can see that they are all
special classes of rather large dimension (compared to the rank of the group).
We give here, for every classical type, the minimal dimension of the Springer
fibre over the dual of any cuspidal unipotent class in good characteristic.

\begin{prop} \label{prop:bound a}
 Assume that $p$ is good for $\bG$ and $Z(\bG)$ is connected. Then there is a
 unique $F$-stable cuspidal class $C_{\min}$ which is contained in the closure
 of every $F$-stable cuspidal class of~$\bG$. Furthermore, the values of
 $a_{\min} := a_{d(C_{\min})}$ for simple classical types are given as follows:
 \begin{enumerate}
  \item[\rm(a)] for type $A_m$: $a_{\min} = \frac{1}{2}m(m+1)$;
  \item[\rm(b)] for type $\tw2A_m$, where $m+1=\binom{s+1}{2}+d$ with
   $0\le d\le s$:
   $$a_{\min} = \frac{1}{6}s(s^2-1)+\frac{1}{2}d(2s+1-d);$$
  \item[\rm(c)] for type $B_m$, where $m=s(s+1)+d$ with $0\le d\le 2s+1$:
   $$a_{\min} = \frac{1}{6}s(s+1)(4s-1)+\begin{cases}
    d(2s+1-d)& \text{ if $d\le s$},\\
    d(4s+2-d)-s(2s+1)& \text{ if $s\le d$};\end{cases}$$
  \item[\rm(d)] for type $C_m$, where $m=s(s+1)+d$ with $0\le d\le 2s+1$:
   $$a_{\min} =\frac{1}{6}s(s+1)(4s-1)+\begin{cases}
    d(2s+2-d)& \text{ if $d\le s+1$},\\
    d(4s+3-d)-(s+1)(2s+1)& \text{ if $s+1\le d$};\end{cases}$$
  \item[\rm(e)] for types $D_m$ and $\tw2D_m$, where $m=s^2+d$ with
   $0\le d\le 2s$:
   $$a_{\min} = \frac{1}{6}s(s-1)(4s+1)+\begin{cases}
    d(2s+1-d)& \text{ if $d\le s$},\\
    d(4s-d)-s(2s-1)& \text{ if $s\le d$}.\end{cases}$$
 \end{enumerate}
\end{prop}

\begin{proof}
Arguing as in \cite[\S3.4]{GM96} we may and will assume that $\bG$ is simple
of adjoint type. For exceptional types there always is a cuspidal family, and
then our claim is in \cite[Thm.~3.3]{GM96}. For type $A_m$ the regular class
$C$ is the only cuspidal class by \cite[Prop.~3.6]{GM96}. Its dual class is
the trivial class, for which $a_{d(C)}=\frac{1}{2}(\dim\PGL_{m+1}-m)=m(m+1)/2$.
\smallskip

For the other classical types,
to any conjugacy class $C$ corresponds a partition $\la$ coming from
the elementary divisors of any element in $C$ in the natural matrix
representation of a classical group isogenous to $\bG$. If $C$ is cuspidal,
then depending on the type of $\bG^F$, the partition $\la$ satisfies the
following properties (see \cite[Prop.~3.6]{GM96}):
\begin{itemize}
 \item type $\tw2A_m$: $\la$ is a partition of $m+1$ into distinct parts;
 \item type $B_{m}$: $\la$ is a partition of $2m+1$ into odd parts, each
  occurring at most twice;
 \item type $C_{m}$: $\la$ is a partition of $2m$ into even parts, each
  occurring at most twice.
 \item type $D_m$ and $\tw2D_m$: $\la$ is a partition of $2m$ into odd parts,
  each occurring at most twice, and $\la\ne(m,m)$ in type $D_m$.
\end{itemize}
(It was claimed erroneously in loc.~cit.~that for type $D_m$, $m$ odd, classes
with label $(m,m)$ are cuspidal, but in fact they do lie in a Levi subgroup of
type $A_{m-1}$.)  \par
First consider type $\tw2A_m$. We claim that the class $C_\mi$ labelled by the
partition $\la=(s+1,\ldots,s-d+2,s-d,\ldots,1)$ is contained in the closure of
all other cuspidal classes. So let $\la$ be a partition with all parts distinct
and first assume that $\la_i\ge\la_{i+1}+2$ and $\la_j\ge\la_{j+1}+2$ for some
$i<j$. Then the partition $\la'$ with
$$\la_i'=\la_i-1,\ \la_{j+1}'=\la_{j+1}+1,$$
and $\la_l'=\la_l$ for all other indices~$l$, is smaller in the dominance order,
hence labels a cuspidal class which is smaller in the partial order of
unipotent classes. Similarly, if $\la_i\ge\la_{i+1}+3$ then $\la'$ with
$$\la_i'=\la_i-1,\ \la_{i+1}'=\la_{i+1}+1,$$
again labels a smaller cuspidal class. Application of these two operations
eventually leads to the label of $C_\mi$, so our claim follows.
\par
In type $C_m$, we claim that the cuspidal class $C_\mi$ such that
$r_i\in\{0,2\}$ for all $i$, and in that range there is at most one~$i$ with
$r_i=r_{i+1}=2$, is containd in the closure of all other cuspidal classes.
Here, $\la^* = (1^{r_1},2^{r_2},\ldots,h^{r_h})$ denotes the conjugate
partition of $\la$ written in exponential notation. Indeed, if $C$ has a label
with $r_i\ge4$ for some $i$, then the partition $\mu$ with $\mu^*$ such that
$$r_{i-1}'=r_{i-1}+2,\ r_i'=r_i-4,\ r_{i+1}'=r_{i+1}+2,$$
labels a smaller cuspidal class. Similarly, if $r_i=r_{i+1}=2$ and
$r_j=r_{j+1}=2$ for some $i<j$, then the partition $\mu$ with
$$r_{i-1}'=r_{i-1}+2,\ r_i'=r_i-2,\ r_{j+1}'=r_{j+1}-2,\ r_{j+2}'=r_{j+2}+2,$$
labels a smaller class.
\par
In type $B_m$, respectively $D_m$ and $\tw2D_m$, let $C_\mi$ denote the
cuspidal class such that removing the biggest part $h=2s+1$ (resp.~$h=2s$) of
$\la^*$ we obtain the label of the minimal cuspidal class in type $C_{m-h}$.
Then as before it is easy to see that $C_\mi$ lies in the closure of all
cuspidal classes. This completes the proof of the first assertion.
\par
To determine $a_{d(C_\mi)}$ note that in all types but $D_m$ and $\tw2D_m$,
the dual of the special class corresponding to $\la$ is the special class
corresponding to the conjugate $\la^*$ of $\la$ (see
\cite[\S12.7 and \S13.4]{Ca}). Then the claimed expression for $a_{d(C_\mi)}$
follows from the centraliser orders given in \cite[\S 13.1]{Ca}: in type
$\tw2A_m$
$$\dim C_\bG(u)=\sum_{i\geq1}(r_i+r_{i+1}+\cdots)^2 - 1$$
for $u\in d(C)$ labelled by $\la^* = (1^{r_1},2^{r_2},\ldots,h^{r_h})$,
in type $C_m$ by
$$\dim C_\bG(u)=\frac{1}{2}\big(\sum_{i\geq1}(r_i+r_{i+1}+\cdots)^2
   +\sum_{i\equiv1(2)}r_i\big),$$
and in type $B_m$ by
$$\dim C_\bG(u)=\frac{1}{2}\big(\sum_{i\geq1}(r_i+r_{i+1}+\cdots)^2
   -\sum_{i\equiv1(2)}r_i\big).$$
\par
In types $D_m$, $\tw2D_m$ there is no elementary operation to deduce the
effect of Lusztig--Spaltenstein
duality on unipotent classes. However one can use \eqref{eq:(1)} to compute
$a_{d(C)}$ from the $a$-value of the Alvis--Curtis dual $\rho^*$ of the unique
special unipotent character with unipotent support~$C$. Let
$$\cS=\begin{pmatrix} a_1& \ldots& a_{2s-1}\\ a_2& \ldots& a_{2s}\end{pmatrix}$$
be the symbol of an irreducible character of $W(D_m)$. It is special if and
only if $a_1\le a_2\le\ldots\le a_{2s}$. By \cite[(5.15)]{MaU}, for example
(noting that the last ``$-$'' sign in that formula should be a ``$+$''),
the $A$-value of the corresponding unipotent principal series character
$\rho_\cS$ is given by
$$A(\cS)=\sum_{i<j}\max\{a_i,a_j\}-\sum_{i=1}^{s-1}\binom{2i}{2}
  -2\sum_{i=1}^{2s}\binom{a_i+1}{2}+m^2,$$
and the $a$-value of the Alvis--Curtis dual $\rho_{\cS^*}=D_G(\rho_\cS)$ of
$\rho_\cS$ is
$$a(\cS^*)=m(m-1)-A(\cS).$$
Now let $\la$ be a partition of $2m$ labelling a cuspidal unipotent class $C$
for $D_m$. By the description of cuspidal unipotent classes, $\la$ has an even
number of non-zero parts, say $\la=(\la_1\ge\la_2\ge\cdots\ge\la_{2l}>0)$. The
Springer correspondent of $C$ is obtained as follows, see \cite[\S2D]{GM00}:
it is labelled by a symbol $\cS$ as above with entries
$a_i=(\la_{2l+1-i}-1)/2+\lfloor\frac{i}{2}\rfloor$, $1\le i\le 2l$. The symbol
$\cS_\mi$ of the character $\rho_\mi$ corresponding to the class $C_\mi$ is
then given by
$$\begin{aligned}
  &\binom{0\quad \ldots\quad 2(s-1)}{1\ \ldots\ 2(s-d)-1\ 2(s-d+1)\ \ldots\ 2s}&
  \quad &\text{for $0\le d\le s$, resp.}\\
  &\binom{0\ \ldots\ 2(2s-d-1)\ 2(2s-d)+1\ \ldots 2s-1}{2\quad \ldots\quad 2s}&
  \quad &\text{for $s\le d\le 2s$}.
\end{aligned}$$
(Here, dots signify that the entries in between are meant to increase in steps
of~2. In particular, all $a_i-a_{i-2}=2$ except at one single position.)
Application of the above formula for $a(\cS_\mi^*)$ yields the assertion.
\end{proof}

\begin{rem}
Theorem~\ref{thm:low a} and Proposition~\ref{prop:bound a} show that unipotent
Brauer characters with ``small'' $a$-value must lie in ``small'' Harish-Chandra
series, that is, in Harish-Chandra series corresponding to Levi subgroups of
small semisimple rank. In particular, cuspidal modules have a large $a$-value
compared to the rank of the group. \par
More precisely, by Proposition~\ref{prop:bound a} the $a$-values of cuspidal
unipotent classes in classical groups grow roughly like $c m^{\frac{3}{2}}$
with the rank~$m$, where $c=\frac{\sqrt{2}}{3}$ for type $\tw2A_m$ and
$c=\frac{2}{3}$ for types $B_m,C_m,D_m$ and $\tw2D_m$. This is the same order
of magnitude as the $a$-value of a cuspidal unipotent character. Since the
latter is a fixed point of Alvis--Curtis duality, this shows that cuspidal
unipotent Brauer characters only occur ``in the lower half'' of the
decomposition matrix of a group of classical type.
\end{rem}

In Table~\ref{tab:low-a} we give the minimal $a$-value for classical groups of
rank at most~10.

\begin{table}[htbp]
\[\begin{array}{c|cccccccccc}
  m& 1& 2& 3& 4& 5& 6& 7& 8& 9& 10\\
\hline
 A_m& 1& 3& 6& 10& 15& 21& 28& 36& 45& 55\\
 \tw2A_m& -& 1& 3& 4& 4& 7& 9& 10& 10& 14\\
  B_m& -& 1& 3& 6& 7& 7& 11& 13& 18& 21\\
  C_m& -& 1& 4& 5& 7& 7& 12& 15& 16& 20\\
  D_m,\tw2D_m& -& -& -& 3& 7& 9& 12& 13& 13& 19\\
\end{array}\]
 \caption{$a_{d(C_{\min})}$ in classical types}  \label{tab:low-a}
\end{table}

\subsection{Induction of cuspidal classes}
Let us recall the construction of induced unipotent classes by
Lusztig--Spaltenstein \cite{LuSp}. Let $\bP = \bL \bU$ be a parabolic subgroup
of $\bG$ with Levi complement $\bL$. Given any unipotent class $C$ of $\bL$,
the variety $C \bU$ is irreducible and consists of unipotent elements only.
Therefore there exists a unique unipotent class $\widetilde C$ of $\bG$ such
that $\widetilde C \cap C\bU$ is dense in $C\bU$. The class $\widetilde C$ is
called the induction of $C$ from $\bL$ to $\bG$. It is shown in \cite{LuSp}
that it depends only on $(C,\bL)$ and not on $\bP$. As in \cite{LuSp} we will
denote it by $\Ind_\bL^\bG(C)$.

\begin{prop}   \label{prop:ind-cusp}
 Assume that $p$ is good for $\bG$ and $Z(\bG)$ is connected. Let $\bL$ be a
 $1$-split Levi subgroup of~$\bG$. If $C$ is an $F$-stable cuspidal unipotent
 class of $\bL$ then $\Ind_\bL^\bG(C)$ is an $F$-stable cuspidal unipotent
 class of~$\bG$.
\end{prop}

\begin{proof}
By standard reductions, we can assume that $\bG$ is simple of adjoint type
(see for example \cite[\S3.4]{GM96}). Also, by transitivity of the induction
of unipotent classes, we can assume without loss of generality that $\bL$ is
a maximal proper $1$-split Levi subgroup of~$\bG$.
Assume that $\bL$ has an untwisted type $A_r$-factor. Its only cuspidal
unipotent class is the regular class, which is obtained by induction of the
trivial class of a maximal torus. By transitivity of induction we may thus
replace this factor by a maximal torus, and then replace that Levi subgroup by
any other maximal $1$-split Levi subgroup of $\bG$ containing it. In
particular, our claim holds for groups of untwisted type $A$.
\par
Thus, if $\bG$ is simple of classical type $X_m$ (where $m$ is the rank), then
we may assume that $\bL$ is a group of type $X_{m-\delta}$, where $\delta = 2$
if $X_m = \tw2A_m$ and $\delta = 1$ otherwise.
Let $\lambda$ be a partition labelling a cuspidal unipotent class $C$ of $\bL$.
It satisfies the properties listed in the beginning of the proof of
Proposition~\ref{prop:bound a}, depending on the type of $\bL$. Under these
conditions, if follows from \cite[\S 7.3]{SpB} that the partition labelling
$\Ind_\bL^\bG(C)$ is $(\lambda_1+2,\lambda_2,\ldots)$, and so the corresponding
class is also cuspidal.
\par
If $\bG$ is simple of exceptional type, then by our previous reduction we are
in one of the cases listed in Table~\ref{tab:ind-cusp-exc}. The induction of
classes is then easily determined using for example the \Chevie\ \cite{ChM}
command {\tt UnipotentClasses} and the fact that induction of classes is just
$j$-induction of the corresponding Springer representations,
see \cite[4.1]{Sp85}.
\end{proof}

\begin{table}[htbp]
\[\begin{array}{c|ccc}
  B_3 & 7 & 51^2 & 3^21 \\
  C_3 & 6 & 42 \\\hline
  F_4 & F_4 & F_4(a_1) & F_4(a_2)\\
\end{array}\]\smallskip
\[\begin{array}{c|ccc}
  D_5 & 91 & 73 & 531^2\\\hline
  E_6 & E_6 & E_6(a_1) & E_6(a_3)\\
\end{array}\]\smallskip
\[\begin{array}{c|ccccc}
  \tw2A_5 & 6 & 51 & 42 & & 321 \\
  \tw2D_4 & 71 & 53 & & 3^21^2\\\hline
  \tw2E_6 & E_6 & E_6(a_1) & D_5 & E_6(a_3) & D_5(a_1)\\
\end{array}\]\smallskip
\[\begin{array}{c|ccccccc}
  D_6 & 11.1 & 93 & 75 & 731^2 & 5^21^2 & 53^21 \\
  E_6 & E_6 & E_6(a_1) & & E_6(a_3) & & & D_4(a_1) \\\hline
  E_7 & E_7 & E_7(a_1) & E_7(a_2) & E_7(a_3) & E_6(a_1) & E_7(a_4) & E_7(a_5)\\
\end{array}\]\smallskip
\[\begin{array}{c|ccccccccc}
  D_7 & 13.1 & 11.3 & 95 & 931^2 & 751^2 & 73^21 & 5^2 31 \\
  E_7 & E_7 & E_7(a_1) & E_7(a_2) & E_7(a_3) & E_6(a_1) & E_7(a_4) & & E_7(a_5)  & A_4+A_1 \\\hline
  E_8 & E_8 & E_8(a_1) & E_8(a_2) & E_8(a_3) & E_8(a_4) & E_8(b_4) & E_8(a_5) & E_8(b_5) &  E_6(a_1)+A_1\\
\end{array}\]
 \caption{Some induced cuspidal classes}  \label{tab:ind-cusp-exc}
\end{table}

\section{Irreducibility of the smallest character in a Harish-Chandra series}
We extend our irreducibility result in \cite[Thm.~A]{DM17}.

\begin{lem}   \label{lem:all types}
 Assume that $\bG$ is simple of adjoint type. Let $\bL\le\bG$ be a $1$-split
 Levi subgroup with a cuspidal unipotent character $\la$.
 Then $a_\la\le a_{d(C)}$ for all $F$-stable cuspidal unipotent classes $C$
 of $\bG$, with equality occurring if and only if $\bL=\bG$.
\end{lem}

\begin{proof}
For classical types, see Proposition~\ref{prop:bound a}. For exceptional types,
the minimal $a_{d(C)}$ and the possible $a_\la$ are easily listed using
\cite[Prop.~3.6]{GM96} and the tables in \cite[\S13]{Ca};
they are given in Table~\ref{tab:low-a-exc}.
\end{proof}

\begin{rem}
For an alternative proof of the lemma based on Proposition~\ref{prop:ind-cusp}
consider the (cuspidal) unipotent support $C_\la$ of $\la$ and the induced
class $\widetilde C_\la = \Ind_\bL^\bG(C_\la)$. Then
$a_\la=a_{d(C_\la)}=a_{C_\la}=a_{\widetilde C_\la}$ by
\cite[Thm.~1.3]{LuSp}.
Moreover, by Proposition~\ref{prop:ind-cusp}, the class $\widetilde C_\la$
is cuspidal. Therefore the minimal cuspidal class $C_\mi$ from
Proposition~\ref{prop:bound a} lies in the closure of $\widetilde C_\la$,
and hence $a_{\la} \leq a_{C_\mi} = a_{d(C_\mi)}$ with equality only when
$C_\mi = \widetilde C_\la$. In that case the unipotent support of $\la$
is self-dual, therefore the character $\la$ is its own Alvis-Curtis dual
and $\bL = \bG$.
\end{rem}

\begin{table}[htbp]
\[\begin{array}{c|cccccccccl}
  & C_\mi& a_{d(C_\mi)}& \text{possible } a_\la\\
\hline
 G_2& G_2(a_1)& 1& 0,1\\
 F_4& F_4(a_3)& 4& 0, 1, 4\\
 E_6,\tw2E_6& D_4(a_1)& 7& 0,3,4,7\\
 E_7& A_4+A_1& 11& 0,3,7,11\\
 E_8& E_8(a_7)& 16& 0,3,7,11,16\\
\end{array}\]
 \caption{Minimal $a_{d(C)}$ in exceptional types}  \label{tab:low-a-exc}
\end{table}

The following extends our result \cite[Prop.~4.3]{DM15} for unitary groups
(which was shown using the known unitriangularity of the decomposition
matrix) to all reductive groups; here instead we use the irreducibility of
cuspidal unipotent characters from \cite[Thm.~A]{DM17}:

\begin{thm}   \label{thm:HC-irr}
 Assume that $p$ is good for $\bG$. Let $\ell\ne p$ be a prime which is good
 for $\bG$ and not dividing $|Z(\bG)^F/Z^\circ(\bG)^F|$. Let $\rho$ be a
 unipotent character of $G$ which has minimal $a$-value in its (ordinary)
 Harish-Chandra series. Then the $\ell$-modular reduction of $\rho$ is
 irreducible.
\end{thm}

\begin{proof}
As in the proof of \cite[Thm.~A]{DM17} we can assume that $\bG$ is simple
of adjoint type.
\par
Let $\bL\le \bG$ be a 1-split Levi subgroup and $\la\in\Irr(L)$ a cuspidal
unipotent character such that $\rho$ lies in the Harish-Chandra series of
$(L,\la)$. Since $\bG$ is adjoint, the centre of $\bL$ is connected
and by \cite[Thm.~A]{DM17} there exists a unique PIM of $L$ containing $\la$.
In addition, it satisfies $\langle\la,P_\la\rangle=1$ and $a_{P_\la} = a_\la$.
\par
Recall that the characters in the Harish-Chandra series of $(L,\la)$ are in
bijection with the irreducible representations of the Hecke algebra
corresponding to the cuspidal pair. Under this bijection, the character
$\rho$ of minimal $a$-value corresponds to the representation that specialises
to the trivial character at $q=1$. Furthermore, the degree of $\rho$ equals
the product of $\la(1)$ with the Poincar\'e polynomial of the Hecke
algebra. In particular, $a_\la = a_\rho$. In addition,
$\langle\rho,\RLG(P_\la)\rangle=1$ by the Howlett--Lehrer comparison theorem.
\par
Let $P_\rho$ denote the (projective) indecomposable summand of $\RLG(P_\la)$
containing $\rho$ (once). Then $a_\rho \geq a_{P_\rho} \geq a_{P_\lambda}$.
Since $a_{P_\la} = a_{\la} = a_\rho$ we deduce that $a_{P_\rho} = a_\rho$.
\par
We claim that $\rho$ does not occur in any other PIM of $G$. For this, assume
that $\rho$ is a constituent of $R_M^G(P)$ for some $1$-split Levi subgroup
$\bM\le \bG$
and some PIM $P$ of $M$. Then $P$ has to have some constituent $\psi\in\Irr(M)$
from the $(L,\la)$-series, whence $a_\la\le a_\psi\le a_\rho=a_\la$. So we have
that $a_\psi=a_\la$ and $\psi$ is the character of $M$ with minimal $a$-value
in the $(L,\la)$-series. If $M$ is proper, then by induction there is exactly
one PIM $P_\psi$ of $M$ containing $\psi$ which is a direct summand
of $R_L^M(P_\la)$; as $P_\rho$ occurs exactly once in $\RLG(P_\la)$, we see
that the summand of $R_M^G(P)=R_M^G(P_\psi)$ containing $\rho$ is just
$P_\rho$.   \par
To conclude, finally assume that $M=G$, so $P$ is a cuspidal PIM of $G$. But
then by \cite[Prop.~2.2]{DM17} it
occurs as a summand of a GGGR $\Gamma_C$ for some $F$-stable cuspidal unipotent
class $C$ of $\bG$. According to Lemma~\ref{lem:all types} we have that
$a_{\Gamma_C}\ge a_\la$, with equality only when $\bL=\bG$. Thus, $\gamma_C$ can
contain $\rho$ only when $\bL=\bG$, and so $\rho=\la$ is cuspidal. But then
our claim is just \cite[Thm.~A]{DM17}.
\end{proof}

\section{Small Harish-Chandra series in even-dimensional orthogonal groups}
\label{sec:orth-even}
We start by considering the two families of even-dimensional orthogonal groups.
As Theorem~\ref{thm:low a} only holds in good characteristic, we shall
restrict ourselves to the case where $q$ is odd. As before, $\ell$ will denote
a prime not dividing $q$.

\subsection{Small cuspidal Brauer characters}
We first determine the cuspidal unipotent Brauer characters with small
$a$-value. For a connected reductive group $\bG$ we set $\bG_\ad:= \bG/Z(\bG)$.

\begin{lem}   \label{lem:cuspDn}
 Let $\bG$ be connected reductive and assume that
 \begin{enumerate}
   \item[\rm(1)] $G_\ad$ is simple of type $A_m(q)$ $(m\ge1)$,
 $D_m(q)$ $(q$ odd, $m\ge4)$ or $\tw2D_m(q)$ $(q$ odd, $m\ge2)$;
   \item[\rm(2)] $\ell$ does not divide $2q|Z(\bG)/Z^\circ(\bG)|$; and
   \item[\rm(3)] if $\ell=3|(q+1)$ then $G_\ad\not\cong\tw2D_3(q)$.
 \end{enumerate}
 Then any cuspidal unipotent Brauer character $\vhi\in\IBr(G)$
 with $a_\vhi\le3$ occurs in Table~\ref{tab:cuspDn}.
\end{lem}

\begin{table}[ht]
\caption{Some cuspidal unipotent Brauer characters in types $A_m$, $D_m$ and $\tw2D_m$}   \label{tab:cuspDn}
\[\begin{array}{c|lc|l}
 G_\ad& \vhi& a_\vhi& \text{condition on $\ell$}\\
\hline
     A_1(q)&  1^2&  1& \ell|(q+1)\\
     A_2(q)&  1^3&  3& \ell|(q^2+q+1)\\
 \tw2D_2(q)& -.1&  2& \ell|(q^2+1)\\
 \tw2D_3(q)& -.2&  3& \ell|(q+1)\\
     D_4(q)&   D_4&  3& \text{always}\\
\hline
\end{array}\]
\scriptsize{The Brauer characters are labelled, via the triangularity of the
decomposition\\ matrix, by the Harish-Chandra labels of the ordinary unipotent
characters.}
\end{table}

\begin{proof}
Let $\vhi\in\IBr(G)$ be as in the statement. By Theorem~\ref{thm:low a},
there exists a cuspidal unipotent class $C$ of $\bG_\ad$ with $a_{d(C)}\le3$.
By Proposition~\ref{prop:bound a}, $G_\ad$ must then be of one of the types
$$A_1(q),\ A_2(q),\ \tw2D_3(q)\cong A_1(q^2),\ \tw2D_3(q)\cong\tw2A_3(q),\
  \tw2D_4(q)\text{ or }D_4(q).$$
Under our assumptions on $\ell$ for all listed groups the unipotent characters
form a basic set for the union of unipotent $\ell$-blocks of both $G$ and
$G_\ad$ \cite{GH91,Ge93}, and the decomposition matrix with respect
to these is uni-triangular (see \cite{Ja90,DM15,GP92}).
Moreover, from the explicit knowledge of these decomposition matrices,
it follows that $\vhi$ must be as in Table~\ref{tab:cuspDn}.
\end{proof}

\subsection{Small Harish-Chandra series}
We can now determine those unipotent Harish-Chandra series that may contain
Brauer characters of small $a$-value.

\begin{lem}   \label{lem:a in A}
 Assume $G_\ad$ is simple of type $A_m(q)$ and $\ell>2$.
 \begin{enumerate}
  \item[\rm(a)] If $m\ge2$ and $\ell|(q+1)$ then $a_\vhi\ge3$ for all unipotent
   Brauer characters $\vhi\in\IBr(G)$ in the $(A_1,1^2)$-Harish-Chandra series.
  \item[\rm(b)] If $m\ge3$, and $\ell|(q^2+q+1)$, then $a_\vhi\ge6$ for all
   unipotent Brauer characters $\vhi\in\IBr(G)$ in the
   $(A_2,1^3)$-Harish-Chandra series.
  \item[\rm(c)] If $m\ge4$ and $\ell|(q+1)$ then $a_\vhi\ge4$ for all unipotent
   Brauer characters $\vhi\in\IBr(G)$ in the
   $(A_1^2,1^2\otimes1^2)$-Harish-Chandra series.
 \end{enumerate}
\end{lem}

\begin{proof}
It follows from the known unipotent decomposition matrices (see
\cite[App.~1]{Ja90}), or from the known distribution into modular
Harish-Chandra series, that $a_\vhi=3$ for $m=2$ in~(a), $a_\vhi=6$ for $m=3$
in~(b), and $a_\vhi\ge4$ for $m=4$ in~(c) respectively.
Since Harish-Chandra induction does not diminish the $a$-value (see for example
\cite[Cor.~8.7]{LuB}), the claim also holds for all larger~$m$.
\end{proof}

\begin{prop}   \label{prop:HC Dn}
 Assume that $n\ge4$ and $\ell{\not|}\,6q$. Then all unipotent Brauer characters
 of $\Spin_{2n}^\pm(q)$, $q$ odd, with $a_\vhi\le3$ lie in one of the
 Harish-Chandra series given in Table~\ref{tab:smallDn}.
\end{prop}

\begin{table}[ht]
\[\begin{array}{c|c|l}
 & (L,\la)& \text{conditions}\\
\hline
D_n(q)& (\emptyset,1),\ (D_4,D_4)& \text{always}\\
      & (A_1,1^2),\  (D_2,-.2)& \ell|(q+1)\\
      & (A_1^2,1^2\otimes 1^2),\ (D_2A_1,-.2\otimes 1^2)& \ell|(q+1),\ n=4\\
\hline
\tw2D_n(q)& (\emptyset,1)& \text{always}\\
      & (A_1,1^2),\  (\tw2D_3,-.2)& \ell|(q+1)\\
      & (\tw2D_2,-.1)& \ell|(q^2+1)\\
      & (A_1^2,1^2\otimes 1^2)& \ell|(q+1),\ n=5\\
      & (A_2,1^3)& \ell|(q^2+q+1),\ n=4\\
\hline
\end{array}\]
\caption{Unipotent Harish-Chandra series in $\tw{(2)}D_n(q)$ with small $a$-value}   \label{tab:smallDn}
\end{table}

\begin{proof}
First assume that $G=\Spin_{2n}^+(q)$. Let $\vhi\in\IBr(G)$ be an
$\ell$-modular unipotent Brauer character of $a$-value at most~3. It lies in
the Harish-Chandra series associated to a cuspidal pair $(L,\la)$, with
$\la\in\IBr(L)$. Note that for Levi subgroups $\bL$ of $\bG$ we have that
$Z(\bL)/Z^\circ(\bL)$ is a 2-group, so that under our assumptions on $\ell$
the unipotent Harish-Chandra series and unipotent cuspidal characters of
$L$ are as in its adjoint quotient. But the latter is a direct
product of simple groups of adjoint type. Thus, since the $a$-value is
additive over outer tensor products of characters we may apply
Lemma~\ref{lem:cuspDn}. By the shape of parabolic subgroups of a Weyl group
of type $D_n$, $L$ is of one of the types
$$A_1^m(q)\ (m\le3),\ D_2(q),\ D_2(q)A_1(q),\ A_2(q), \text{ or } D_4(q),$$
$\la$ is a cuspidal unipotent Brauer character of $L$ as in
Table~\ref{tab:cuspDn}, and $\ell$ is as given there. We consider these
Harish-Chandra series $(L,\la)$ in turn.   \par
Assume that $\ell|(q^2+q+1)$, so $L$ is of type $A_2(q)$. For $n\ge4$ the
group $G$ contains a Levi subgroup of type $A_3$, and so by
Lemma~\ref{lem:a in A}(b) the $(A_2,1^3)$-series of $G$ only contains Brauer
characters of $a$-value at least~$6$.
\par
It remains to consider the case that $\ell|(q+1)$. By Lemma~\ref{lem:a in A}
neither the $A_1^2$-series nor the $D_2A_1$-series can contribute if $n\ge5$,
since these Levi subgroups are contained in a Levi subgroup of $G$ of
type~$A_4$, resp.~of type $D_2A_2$. Similarly, $A_1^3$ (which occurs for
$n\ge6$) is contained in a Levi subgroup of type $A_3A_1$ and so
the $a$-values in that series are at least~4.
\par
Now let $G=\Spin_{2n}^-(q)$ and $\vhi$ as in the statement. So by
Theorem~\ref{thm:low a}, $\vhi$ lies in a Harish-Chandra series $(L,\la)$
with $a_\la\le3$. Thus, by Lemma~\ref{lem:cuspDn}, $L$ is of one of the types
$$A_1^m(q)\ (m\le3),\ \tw2D_2(q),\ A_2(q), \text{ or } \tw2D_3(q),$$
and $\la\in\IBr(L)$ is a cuspidal Brauer character of $L$ as in
Table~\ref{tab:cuspDn}. We consider the contributions by these Harish-Chandra
series in turn. For $\ell|(q^2+q+1)$ the known Brauer trees show that for $n=4$
there is one Brauer character in the $(A_2,1^3)$-series with $a$-value~3 ,
but for $n=5$ (and hence for all larger~$n$) all $a$-values are at least~6 .
\par
Now consider the case when $\ell|(q+1)$. Since $\tw2D_6$ contains a Levi
subgroup of type $A_4$ if $n\ge6$ we see by Lemma~\ref{lem:a in A} that the
Harish-Chandra series of type $A_1^2$ cannot contain $\vhi$ in that case.
Similarly $A_1^3$ lies in a Levi subgroup of type $A_3A_1$ and so the
$A_1^3$-series does not contain $\vhi$.
\end{proof}

\subsection{A triangular subpart of the decomposition matrix}
We give an approximation to the $\ell$-modular decomposition matrix of
$\Spin_{2n}^\pm(q)$ for certain unipotent Brauer characters of $a$-value at
most~3. For simplicity, in view of Table~\ref{tab:smallDn} we only consider
primes $\ell$ not dividing $q+1$.
Recall that unipotent characters of $\Spin_{2n}^+(q)$ are labelled by symbols
of rank~$n$ and defect congruent to~0 modulo~4 (see \cite[\S13.8]{Ca}).

\begin{thm}   \label{thm:decD+}
 Let $G=\Spin_{2n}^+(q)$ with $q$ odd and $n\ge5$, and $\ell>5$ a prime not
 dividing $q(q+1)$. Then the first eight rows of the decomposition matrix of
 the unipotent $\ell$-blocks of $G$ are approximated from above by
 Table~\ref{tab:Spin+}, where $k=n-5$ and
 $$(a,b,c,d,e)=\begin{cases}
                         (1,0,0,0,0)& \text{when $\ell|(q^2+q+1)$},\cr
                         (0,1,0,0,0)& \text{when $\ell|(q^2+1)$},\cr
                         (0,0,1,0,0)& \text{when $\ell|(q^4+q^3+q^2+q+1)$},\cr
                         (0,0,0,1,0)& \text{when $\ell|(q^2-q+1)$},\cr
                         (0,0,0,0,1)& \text{when $\ell|(q^4+1)$},\cr
                         (0,0,0,0,0)& \text{otherwise}.\cr
 \end{cases}$$
\end{thm}

\begin{table}[ht]
{\small\[\begin{array}{l|c|ccc|c|cccc}
  \quad\rho& a_\rho\cr
 \hline
        \binom{n}{0}& 0& 1\cr
      \binom{n-1}{1}& 1& e\pl k& 1\cr
    \binom{1,n}{0,1}& 2& c\pl k& .& 1\cr
 \hline
        \binom{3}{2}& 2& b\pl d\pl ke\pl\binom{k}{2}& b\pl k& .&  1\cr
 \hline
\binom{0,1,2,n-1}{-}& 3& .& .& .&  .& 1\cr
  \binom{0,n-1}{1,2}& 3& ke\pl\binom{k}{2}& e\pl k& .&  .& .& 1\cr
  \binom{1,n-1}{0,2}& 3& k(c\pl e\pl k\mn1)& b \pl k& d\pl k& b& .& .& 1\cr
  \binom{2,n-1}{0,1}& 3& b\pl kc\pl\binom{k}{2}& .& a\pl k& .& .& .& .& 1\cr
 \hline
 \text{HC-series}& & ps& ps& ps& ps& D_4& ps& ps& ps\cr
\end{array}\]}
\caption{Approximate decomposition matrices for $\Spin_{2n}^+(q)$, $n\ge5$}   \label {tab:Spin+}
\end{table}

\begin{proof}
Let $\vhi\in\IBr(G)$ be an $\ell$-modular constituent of one of the eight
unipotent characters $\rho_1,\ldots,\rho_8$ listed in Table~\ref{tab:Spin+}
(in that order). Then by Brauer reciprocity, $\rho_i$ is a constituent of the
projective cover of $\vhi$. The degree formula \cite[\S13.8]{Ca} shows that
$a_{\rho_i}\le3$ for all $i$. Thus, by Proposition~\ref{prop:HC Dn}, $\vhi$
lies in one of the Harish-Chandra series $(L,\la)$ given in
Table~\ref{tab:smallDn}. We consider these in turn.
\par
The Hecke algebra for the principal series is the Iwahori--Hecke algebra
$\cH(D_n;q)$ of type $D_n$, whose $\ell$-decomposition
matrix is known to be uni-triangular with respect to a canonical basic set
given by FLOTW bipartitions, see \cite[Thm.~5.8.19]{GJ11}. All of the $\rho_i$
apart from $\rho_5$ are contained
in this basic set, for all relevant primes $\ell$. Thus the corresponding part
of the decomposition matrix is indeed lower triangular. The given upper bounds
on the entries in this part of the decomposition matrix are obtained by
Harish-Chandra inducing the PIMs from $\Spin_8^+(q)$.
For this group we can consider each cyclotomic factor in turn. The Sylow
$\ell$-subgroups for primes $\ell>3$ dividing $q^2+q+1,q^4+q^3+q^2+q+1,q^2-q+1$
or $q^4+1$ are cyclic, and the Brauer trees of unipotent blocks are easily
determined. The unipotent decomposition matrix for primes $3<\ell|(q-1)$ is
the identity matrix by a result of Puig, see \cite[Thm.~23.12]{CE}. Finally,
for $(q^2+1)_\ell>5$ the unipotent decomposition matrix of $\Spin_8^+(q)$ was
determined in \cite[Thm.~3.3]{DM16}.  \par
Next consider the Harish-Chandra series of the ordinary cuspidal unipotent
character $\la$ of a Levi subgroup $L$ of type $D_4$. Then $\RLG(\la)$
only contains unipotent characters in the ordinary Harish-Chandra series of
type $D_4$, hence among $\rho_1,\ldots,\rho_8$ only $\rho_5$, just once.
So there is exactly one PIM in this series involving one of the $\rho_i$,
namely the projective cover of $\rho_5$.
\end{proof}

We now consider the similar question for the non-split orthogonal groups
$\Spin_{2n}^-(q)$. Recall that its unipotent characters are labelled by
symbols of rank~$n$ and defect congruent to~2 modulo~4. Since all considered
characters lie in the principal series, they can also be indexed by
irreducible characters of the Weyl group of type $B_{n-1}$, hence by suitable
bipartitions of $n-1$, which we will do here.

\begin{thm}   \label{thm:decD-}
 Let $G=\Spin_{2n}^-(q)$ with $q$ odd and $n\ge5$, and $\ell>5$ a prime not
 dividing $q(q+1)$. Then the first eight rows of the decomposition matrix of
 the unipotent $\ell$-blocks of $G$ are approximated from above by
 Table~\ref{tab:Spin-}, where $k=n-5$ and
 $$(a,b,c,d,e)=\begin{cases}
                         (1,0,0,0,0)& \text{when $\ell|(q^2+q+1)$},\cr
                         (0,1,0,0,0)& \text{when $\ell|(q^2+1)$},\cr
                         (0,0,1,0,0)& \text{when $\ell|(q^2-q+1)$},\cr
                         (0,0,0,1,0)& \text{when $\ell|(q^4+1)$},\cr
                         (0,0,0,0,1)& \text{when $\ell|(q^4-q^3+q^2-q+1)$},\cr
                         (0,0,0,0,0)& \text{otherwise}.\cr
 \end{cases}$$
 The eight corresponding PIMs all lie in the principal series, except that
 for $\ell|(q^2+1)$ the eighth (and the seventh when $n=4$) is in the
 $\tw2D_2$-series.
\end{thm}

\begin{table}[ht]
{\small\[\begin{array}{l|c|ccc|c|cccc}
  \qquad\rho& a_\rho\cr
\hline
   (n-1;-)& 0& 1\cr
 (n-2,1;-)& 1& b\pl k& 1\cr
    (n-2;1)& 2& e\pl k& .& 1\cr
\hline
 (n-3,2;-)& 2& a\pl kb\pl \binom{k}{2}& k& .& 1\cr
\hline
 (n-3,1^2;-)& 3& kb\pl \binom{k}{2}& b\pl k& .& .& 1\cr
  (n-3,1;1)& 3& k(b\pl e\pl k\mn1)& d\pl k& a\pl k& b& .& 1\cr
    (n-3;2)& 3& d\pl ke\pl \binom{k}{2}& .& c\pl k& .& .& .& 1\cr
   (-;n-1)& 3& b& .& .& .& .& .& .& 1\cr
\hline
\end{array}\]}
\caption{Approximate decomposition matrices for $\Spin_{2n}^-(q)$, $n\ge5$}
    \label {tab:Spin-}
\end{table}

\begin{proof}
This is completely analogous to the proof of Theorem~\ref{thm:decD+}.
The eight unipotent characters $\rho_1,\ldots,\rho_8$ displayed in
Table~\ref{tab:Spin-} all have $a$-value at most~3. Thus, arguing as in the
proof of Theorem~\ref{thm:decD+} and using Table~\ref{tab:smallDn} we see that
any $\ell$-modular constituent $\vhi$ of the $\rho_i$ either lies in the
principal series, or $\ell|(q^2+1)$ and $\vhi$ lies in the $\tw2D_2$-series.
\par
For $n=4$ the Sylow $\ell$-subgroups of $G$ are cyclic for all relevant primes
$\ell$, and the claim follows from the known Brauer trees. So now assume that
$n\ge5$.
By \cite[Thm.~5.8.13]{GJ11} the Iwahori-Hecke algebra $\cH(B_{n-1};q^2;q)$
for the principal series has a canonical basic set indexed by suitable Uglov
bipartitions. For $n\ge5$ the bipartitions indexing $\rho_1,\ldots,\rho_7$,
as well as $\rho_8$ when $\ell{\not|}(q^2+1)$, are Uglov, so these characters
lie in the basic set. Hence, there are at most eight PIMs in the principal
series (resp.~seven when $\ell\mid(q^2+1)$) involving one of the $\rho_i$;
since the unipotent characters form a basic set for the unipotent blocks by
\cite{GH91}, it must be exactly eight (resp.~seven).
\par
Finally, assume that $\ell|(q^2+1)$ and consider the series of the cuspidal
Brauer character $(-;1)$ of a Levi subgroup of type $\tw2D_2$. It follows from
the known Brauer trees that there are two PIMs of $\tw2D_4(q)$ in the
$\tw2D_2$ Harish-Chandra series, with unipotent parts $(-;3)$ and $(1;2)$
respectively. Harish-Chandra induction to $\tw2D_5(q)$ shows that there is
exactly one PIM in this series containing one of the $\rho_i$, namely
$\rho_8$ with label $(-;4)$ (see \cite[Thm.~7.1]{DM16}). The Harish-Chandra
induction of $(-;4)$ from $\tw2D_5(q)$ to $\tw2D_n(q)$ only contains $\rho_8$
among our $\rho_i$, so there is at most one PIM in this series that contributes
to the first eight rows of the decomposition matrix.
\end{proof}

\begin{rem}
To complete the determination of the part of the decomposition matrix displayed
in Table~\ref{tab:Spin-} it would be enough to compute the corresponding part
of the decomposition matrix of the Hecke algebra $\cH(B_{n-1};q^2;q)$. When
$\ell$ is large, this can be done using canonical basis elements of Fock
spaces (see \cite[\S 6.4]{GJ11}).
\end{rem}

\subsection{Unipotent Brauer characters of low degree}

As a consequence we can determine those irreducible Brauer characters of low
degree which occur as constituents of $\rho_1,\ldots,\rho_8$. For this we need
the $\ell$-modular decomposition of the smallest three unipotent characters
$\rho_1=1_G$, $\rho_2$ and $\rho_3$; here, for an integer $m$, we set
$$\kappa_{\ell,m}:= \left\{ \begin{array}{ll}
		1 & \text{if } \ell|m, \\
		0 & \text{otherwise.} \end{array} \right.$$

\begin{prop}[Liebeck]   \label{prop:decmatDn}
 Let $G=\Spin_{2n}^\pm(q)$, $n\ge3$ and let $\ell$ be a prime not dividing
 $q(q+1)$. The $\ell$-modular decomposition of $\rho_2,\rho_3$ is given by
 \[\begin{array}{l|lllll}
       1& 1\\
  \rho_2& a& 1\\
  \rho_3& b& .& 1\\
 \end{array}\]
 where $a=\kappa_{\ell,q^{n-1}\pm1}$, $b=\kappa_{\ell,q^n\mp1}$.
\end{prop}

\begin{proof}
First note that if $\ell\nmid q^{n-1}\pm1$ (resp. $\ell\nmid q^n\mp1$)
then $\rho_3$ (resp. $\rho_2$) does not lie in the principal block (see
\cite[\S 13]{FS}).

The unipotent characters $\rho_2,\rho_3$ are the two non-trivial constituents
of the Harish-Chandra induction of the trivial character from a Levi subgroup
of type $\Spin_{2n-2}^\pm(q)$. This induced character is the permutation
character of the rank~3 permutation module of $G$ on singular 1-spaces.
The claim thus follows from Liebeck \cite[Thm.~2.1 and pp.14/15]{Li86}.
Indeed, in the case~(1) of \cite[p.~10]{Li86} there are obviously just three
$\ell$-modular
composition factors of that permutation module, so all $\rho_i$ remain
irreducible. In the case~(2) there are four composition factors, and those
in head and socle clearly are trivial, so one of the $\rho_i$ is reducible.
Which one that is follows from the block distribution.
\end{proof}

Note that the zeroes in the decomposition matrix in
Proposition~\ref{prop:decmatDn} also follow from
Tables~\ref{tab:Spin+} and~\ref{tab:Spin-}, and the multiplicity of the
trivial character in the $\ell$-modular reductions of $\rho_2$ and $\rho_3$
could also easily be computed from the corresponding Hecke algebras.

\begin{cor}   \label{cor:boundD}
 Under the hypotheses of Theorem~\ref{thm:decD+} resp.~\ref{thm:decD-}
 assume that $\vhi\in\IBr(G)$ is a constituent of one of the unipotent
 characters $\rho_i$, $1\le i\le 8$, in Table~\ref{tab:Spin+}
 resp.~\ref{tab:Spin-}. If
 $$\vhi(1)<\begin{cases}
   q^{4n-10}& \text{for }G=\Spin_{2n}^+(q), \text{ resp.}\\
   q^{4n-10}-q^9& \text{for }G=\Spin_{2n}^-(q),\end{cases}$$
 then $i\le3$.
 In particular, $\vhi$ is as described in Proposition~\ref{prop:decmatDn}.
\end{cor}

\begin{proof}
Let $\vhi_j$, $1\le j\le 8$, denote the Brauer character of the head (and
socle) of the PIM given by column~$j$ of the approximate decomposition matrices
in Table~\ref{tab:Spin+}, respectively~\ref{tab:Spin-}. By the proven partial
triangularity, any constituent $\vhi$ of one of the $\rho_i$ is among the
$\vhi_j$, $1\le j\le 8$.  \par
Now for each $i$, $4\le i\le 8$, a lower bound for $\vhi_i(1)$ can be computed
as follows: we certainly have $\vhi_j(1)\le\rho_j(1)$ for $1\le j\le i-1$.
Then from the $i$th row of the given approximation of the decomposition matrix
with that vector of upper bounds we obtain a lower bound for $\vhi_i(1)$. This
is a polynomial in~$q$ (depending on $n$) of degree at least~$4n-10$, and
subtracting the bound $B$ in the statement we obtain a polynomial of positive
degree and positive highest coefficient. It is now a routine calculation to
see that this difference is positive for all relevant $i,q,n$, so
$\vhi_i(1)\ge B$ for $i\ge4$.
\end{proof}

\section{Small Harish-Chandra series in odd-dimensional orthogonal groups}
\label{sec:orth-odd}
We now consider the odd-dimensional spin groups $\Spin_{2n+1}(q)$. Again we
assume that $q$ is odd in order to be able to apply the results from
Section~\ref{sec:bound}.

\subsection{Small Harish-Chandra series}

Again we first determine cuspidal modules with small $a$-value.

\begin{lem}   \label{lem:cuspBn}
 Assume that $G_\ad$ is simple of type $B_m(q)$ $(q$ odd, $m\ge2)$,
 and let $\ell$ be a prime not dividing $2q$. If $\vhi\in\IBr(G)$ is a
 cuspidal unipotent Brauer character with $a_\vhi\le4$ then $\vhi$ occurs in
 Table~\ref{tab:cuspBn}.
\end{lem}

\begin{table}[ht]
\[\begin{array}{c|lc|l}
 G_\ad& \vhi& a_\vhi& \text{condition}\\
\hline
 B_2(q)&   B_2& 1& \text{always}\\
       & -.1^2& 4& \ell|(q+1)(q^2+1)\\
 B_3(q)& 1^3.-& 4& \ell|(q+1)\\
       & B_2:1^2& 4& \ell|(q+1)(q^2-q+1)\\
\hline
\end{array}\]
\caption{Cuspidal unipotent Brauer characters in type $B_m$}   \label{tab:cuspBn}
\end{table}

\begin{proof}
Let $\vhi\in\IBr(G)$ be as in the statement. By Theorem~\ref{thm:low a},
the projective cover of $\vhi$ is a direct summand of $\Gamma_C^G$ for a
cuspidal unipotent class $C$ of $G_\ad$ with $a_{d(C)}\le4$. By
Proposition~\ref{prop:bound a} this implies $m\in\{2,3\}$.
For these groups the unipotent characters form a basic set for the union of
unipotent $\ell$-blocks of $G$ \cite{Ge93}, and the decomposition matrix with
respect to these is uni-triangular (see \cite{HN14}). From the explicit
knowledge of these decomposition matrices, it follows that $\vhi$ must be as
in Table~\ref{tab:cuspBn}.
\end{proof}

\begin{prop}   \label{prop:HC Bn}
 Assume that $n\ge2$ and $\ell{\not|}\,q(q+1)$. Then the unipotent Brauer
 characters of $\Spin_{2n+1}(q)$, $q$ odd, with $a_\vhi\le4$ lie in a
 Harish-Chandra series as given in Table~\ref{tab:smallBn}.
\end{prop}

\begin{table}[ht]
\[\begin{array}{c|l}
 (L,\la)& \text{conditions}\\
\hline
 (\emptyset,1),\ (B_2,B_2)& \text{always}\\
 (B_2,-.1^2)& \ell|(q^2+1)\\
 (A_2,1^3)& \ell|(q^2+q+1),\ n=3\\
 (B_2A_2,B_2\otimes 1^3)& \ell|(q^2+q+1),\ n=5\\
 (B_3,B_2:1^2)& \ell|(q^2-q+1),\ n=3\\
\hline
\end{array}\]
\caption{Unipotent Harish-Chandra series in $B_n(q)$ with small $a$-value}   \label{tab:smallBn}
\end{table}

\begin{proof}
Let $\vhi\in\IBr(G)$ be an $\ell$-modular unipotent Brauer character with
$a_\vhi\le 4$. So $\vhi$ lies in the Harish-Chandra series associated to a
cuspidal pair $(L,\la)$ with $a_\la\le4$. By Theorem~\ref{thm:low a} and
Lemmas~\ref{lem:cuspDn} and~\ref{lem:cuspBn}, $L$ is either a torus or of one
of the types
$$A_2(q),\ B_2(q),\ B_2(q)A_2(q),\text{ or } B_3(q),$$
with $\la$ the corresponding outer tensor product of the cuspidal Brauer
characters given in Tables~\ref{tab:cuspDn} and~\ref{tab:cuspBn}. (Note that
Levi subgroups $\bL$ of $\bG$ have $|Z(\bL)/Z^\circ(\bL)|\le2$ and so one
can invoke both Lemma \ref{lem:cuspDn} and Lemma \ref{lem:cuspBn})
We deal with the various series in turn. First note that a Levi subgroup of
$G$ of type $A_2$ is contained in one of type $A_3$ when $n\ge4$, and so by
Lemma~\ref{lem:a in A}(b) the $A_2$-series (for $\ell|(q^2+q+1)$) can only
contribute when $n=3$. Similarly, the $B_2A_2$-series of $\la=B_2\otimes 1^3$
can only contribute when $n=5$. From the known Brauer trees it follows that
for $n=4$ all $\ell$-modular Brauer characters in the $(B_3,B_2:1^2)$-series
for $\ell|(q^2-q+1)$ have $a$-value at least~6.
\end{proof}

For primes $\ell|(q+1)$ there are many further modular HC-series with
$a$-value at most~4; we will discuss those elsewhere.

\subsection{A triangular subpart of the decomposition matrix}
We give an approximation to the $\ell$-modular decomposition matrix of
$\Spin_{2n+1}(q)$ for certain unipotent Brauer characters of $a$-value at
most~4. Again we only consider primes $\ell$ not dividing $q+1$. The induction
basis is given by the cases $n=3$ computed in \cite{HN14} and the case $n=4$
treated below. Recall that unipotent characters of $\Spin_{2n+1}(q)$ are
labelled by symbols of rank~$n$ and odd defect (see \cite[\S13.8]{Ca}).

\begin{thm}   \label{thm:decB}
 Let $G=\Spin_{2n+1}(q)$ with $q$ odd and $n\ge4$, and $\ell>5$ a prime not
 dividing $q(q+1)$. Then the first fourteen rows of the decomposition matrix of
 the unipotent $\ell$-blocks of $G$ are approximated from above by
 Table~\ref{tab:SpinOdd}, where $k=n-4$ and
 $$(a,b,c,d)=\begin{cases}
                           (1,0,0,0)& \text{when $\ell|(q^2+q+1)$},\cr
                           (0,1,0,0)& \text{when $\ell|(q^2+1)$},\cr
                           (0,0,1,0)& \text{when $\ell|(q^2-q+1)$},\cr
                           (0,0,0,1)& \text{when $\ell|(q^4+1)$},\cr
                           (0,0,0,0)& \text{otherwise}.\cr
 \end{cases}$$
 The Harish-Chandra series are indicated in the last line; here the twelfth
 PIM is in the principal series unless $\ell|(q^2+1)$.
\end{thm}

\begin{table}[htbp]
{\small\[\begin{array}{l|c|c|cccc|cccc|c|cccc}
  \quad\rho& a_\rho\cr
\hline
    \binom{n}{-}& 0& 1\cr
\hline
\binom{0,1,n}{-}& 1& .& 1\cr
  \binom{0,1}{n}& 1& .& .& 1\cr
  \binom{1,n}{0}& 1& b\pl k& .& .& 1\cr
  \binom{0,n}{1}& 1& d\pl k& .& .& .& 1\cr
\hline
\binom{0,2,n-1}{-}& 2& .& k& .& .& .& 1\cr
  \binom{0,2}{n-1}& 2& b& .& b\pl k& .& .& .& 1\cr
  \binom{2,n-1}{0}& 2& a\pl kb\pl \binom{k}{2}& .& .& k& .& .& .& 1\cr
  \binom{0,n-1}{2}& 2& c\pl kd\pl \binom{k}{2}& .& .& .& b\pl k& .& .& .& 1\cr
\hline
  \binom{1,n-1}{1}& 3& k(b\pl d\pl k\mn1)& .& .& c\pl k& a\pl k& .& .& .& .& 1\cr
\hline
\binom{0,1,2,n}{1}& 4& .& d\pl k& .& .& .& b& .& .& .& .& 1\cr
\binom{0,1,2}{1,n}& 4& .& .& b\pl k& .& .& .& b& .& .& .& .& 1\cr
\binom{1,2,n}{0,1}& 4& kb\pl \binom{k}{2}& .& .& b\pl k& .& .& .& .& .& .& .& .& 1\cr
\binom{0,1,n}{1,2}& 4& kd\pl \binom{k}{2}& .& .& .& d\pl k& .& .& .& .& .& .& .& .& 1\cr
\hline
 \text{HC-series}& & ps& B_2& ps& ps& ps& B_2& ps& ps& ps& ps& B_2& ps/.1^2& ps& ps\cr
\end{array}\]}
\caption{Approximate decomposition matrices for $\Spin_{2n+1}(q)$, $n\ge4$}   \label {tab:SpinOdd}
\end{table}

\begin{proof}
Let first consider the case $n=4$. For primes $\ell>3$ dividing one of $q^2+q+1,q^2-q+1$ or
$q^4+1$, the Sylow $\ell$-subgroups of $G=\Spin_9(q)$ are cyclic, and the
Brauer trees of unipotent blocks are easily determined. In particular the
decomposition matrix is triangular in those cases. The unipotent decomposition
matrix for primes $3<\ell|(q-1)$ is the identity matrix by
\cite[Thm.~23.12]{CE}. Finally, for $\ell|(q^2+1)$ projective
characters of $G$ with the indicated decomposition into ordinary characters
can be obtained by Harish-Chandra inducing PIMs from a Levi subgroup of type
$\Spin_7(q)$, for which Sylow $\ell$-subgroups are cyclic. Harish-Chandra
restriction then shows that all of the non-zero entries "b" under the diagonal
have to be positive. This construction also gives the claim about the modular
Harish-Chandra series. (See also \cite[Thm.~8.2]{DM15} for an analogous
argument for $\Sp_8(q)$.)
\par
Now assume $n\ge5$. All character $\rho_1,\ldots,\rho_{14}$ occurring in the
table have $a$-value at most~4. So as before, their $\ell$-modular constituents
must lie in one of the Harish-Chandra series listed in Table~\ref{tab:smallBn}.
From the known Brauer trees it follows that for $n=5$ none of our characters
lies in the $(B_2A_2,B_2\otimes 1^3)$-series when $\ell|(q^2+q+1)$, so this
will not contribute. Harish-Chandra inducing the corresponding columns of the
approximate unipotent decomposition matrix from a Levi subgroup of semisimple
type $\Spin_{2n-1}(q)$ to $G$ gives the columns stated in
Table~\ref{tab:SpinOdd}, with the induction basis given by the case $n=4$
treated before. We just obtain one column with a character in the
$(B_2,.1^2)$-series, which contains exactly one of the $\rho_i$, so there is
at most one PIM in that series in the range of our table. As in the case $n=4$
we find three projectives containing one of the $\rho_i$ above the cuspidal
unipotent character of $B_2$. \par
All characters marked "ps" lie in the canonical basic set given by FLOTW
bipartitions of the Hecke algebra $\cH(B_n;q;q)$, whose $\ell$-decomposition
matrix is uni-triangular, see \cite[Thm.~5.8.5]{GJ11}.
Similarly, according to \cite[Thm.~5.8.2]{GJ11} the relative Hecke algebra
$\cH(B_{n-2};q^3;q)$ of the cuspidal unipotent character of $B_2(q)$ in
$G$ has a canonical basic set given by FLOTW bipartitions and all characters
marked "$B_2$" lie in that basic set. As the ordinary cuspidal unipotent
character of $B_2$ is irreducible modulo all primes and reduction stable,
the decomposition matrix of that Hecke algebra embeds into the decomposition
matrix of $G$ by Dipper's theorem, so the corresponding part of our decomposition
matrix is indeed lower triangular.
\end{proof}

\subsection{Unipotent Brauer characters of low degree}

We now describe the decomposition of the first five unipotent characters
from Theorem~\ref{thm:decB}.

\begin{prop}   \label{prop:decmatBn}
 Let $G=\Spin_{2n+1}(q)$, $n\ge2$. Assume that  $\ell$ is prime to $2q(q+1)$.
 Then the $\ell$-modular decomposition matrix of $\rho_1=1_G,\ldots,\rho_5$
 from Table~\ref{tab:SpinOdd} is given as follows:
 \[\begin{array}{l|lllll}
       1& 1\\
  \rho_2& .& 1\\
  \rho_3& .& .& 1\\
  \rho_4& a& .& .& 1\\
  \rho_5& b& .& .& .& 1\\
 \end{array}\]
 where $a=\kappa_{\ell,q^n-1}$ and $b=1-a$ if $\ell|(q^{2n}-1)/(q-1)$, and
 $a=b=0$ otherwise.
\end{prop}

\begin{proof}
The claim holds when $n=3$ by \cite{HN14}. For $n\ge4$, all the zeroes in the
table are correct by Table~\ref{tab:SpinOdd}, so we only need to discuss the
last two entries in the first column. \par
The entries $a,b$ are determined as follows: The unipotent characters
$\rho_4,\rho_5$ are constituents of the rank 3 permutation module of $G$ on
singular 1-spaces, and thus the claim for $\ell{\not|}(q^{2n}-1)/(q-1)$
follows from \cite[Thm.~2.1]{Li86}. Moreover, else we have that $a+b=1$.
If $\ell|(q^n-1)$ then $\rho_5$ does not lie in the
principal block (see \cite[\S12]{FS}), so that certainly $b=0$, and hence $a=1$.
On the other hand, if $\ell{\not|}(q^n-1)$, but $\ell|(q^{2n}-1)/(q-1)$,
then $\rho_4$ does not lie in the principal block, whence $a=0,b=1$.
\end{proof}

We obtain the following analogue of Corollary~\ref{cor:boundD}:

\begin{cor}   \label{cor:boundB}
 Under the hypotheses of Theorem~\ref{thm:decB} assume that $\vhi\in\IBr(G)$
 is an $\ell$-modular constituent of one of the unipotent characters
 $\rho_i$, $1\le i\le 14$, in Table~\ref{tab:SpinOdd}.
 If $\vhi(1)<\hlf q^{4n-6}-q^{3n-3}$ then $i\le5$.
 In particular, $\vhi$ is as described in Proposition~\ref{prop:decmatBn}.
\end{cor}


\end{document}